\documentclass[11pt]{amsart}
\usepackage{amsmath,amssymb,amsthm,enumerate,hyperref,color}
\usepackage{amsfonts}
\usepackage{latexsym}
\usepackage{graphicx}
\usepackage{layout}
\newtheorem{theorem}{Theorem}[section]

\newtheorem{proposition}[theorem]{Proposition}
\newtheorem{lemma}[theorem]{Lemma}
\newtheorem{corollary}[theorem]{Corollary}
\newtheorem{remark}[theorem]{Remark}

\newcommand{\sezione}[1]{\section{#1}\setcounter{equation}{0}}

\newcommand{\nor}{\Arrowvert}

\def\up{u_p}

\def\e{{\epsilon}}
\def\di12{\mathcal{D}^{1,2}(\R^n)}

\def\d{\delta}
\def\D{\Delta}
\def\l{{\lambda}}

\def\a{{\alpha}}
\def\b{{\beta}}
\def\g{{\gamma}}

\def\wth{{w_{\theta}}}

\def\de{\partial}

\newcommand{\R}{\mathbb{R}}

\newcommand{\N}{\mathbb{N}}

\newcommand{\cC}{{\mathcal C}}

\newcommand{\cK}{{\mathcal K}}

\newcommand{\cO}{{\mathcal O}}

\newcommand{\cS}{{\mathcal S}}

\begin{document}
\title[Separation of branches]{Separation of branches of $O(N-1)$-invariant solutions for a semilinear elliptic equation}
\author[F.~Gladiali]{Francesca Gladiali} 
\thanks{The author is supported by Gruppo Nazionale per
l'Analisi Matematica, la Probabilit\'a e le loro Applicazioni (GNAMPA)
of the Istituto Nazionale di Alta Matematica (INdAM) and by PRIN-2012-grant ``Variational and perturbative aspects of nonlinear differential problems''.}
\address{Matematica e Fisica, Polcoming, Universit\`a di Sassari, via Piandanna 4, 07100 Sassari, Italy. \texttt{fgladiali@uniss.it}}

\begin{abstract}
\noindent We consider the problem
\begin{equation}\nonumber\label{0}
\left\{\begin{array}{ll}
-\Delta u=u^p +\l u&\hbox{ in }A\\
u>0 &\hbox{ in }A\\
u=0  &\hbox{ on }\de A
\end{array}\right.
\end{equation}
where $A$ is an annulus in $\R^N$, $N\geq 2$ , $p\in(1,+\infty)$ and $\l\in
(-\infty,0]$. Recent results, \cite{GGPS}, ensure that there exists a
  sequence $\{p_k\}$ of exponents ($p_k\to +\infty$) at which a nonradial
  bifurcation from the radial solution occurs. 
Exploiting the properties of $O(N-1)$-invariant spherical harmonics, we introduce two suitable cones $\mathcal{K}^1$ and $\mathcal{K}^2$ of $O(N-1)$-invariant functions that allow to separate the branches of bifurcating solutions from the others, getting the unboundedness of these branches.
\end{abstract}

\maketitle
\sezione{Introduction}
In this paper we consider the problem
\begin{equation}\label{1}
\left\{\begin{array}{ll}
-\Delta u=u^p +\l u&\hbox{ in }A\\
u>0 &\hbox{ in }A\\
u=0  &\hbox{ on }\de A
\end{array}\right.
\end{equation}
where $A$ is an annulus of $\R^N$, i.e. $A:=\{x\in \R^N\, :\,
a<|x|<b\}$, $b>a>0$, $N\geq 2$, $p\in(1,+\infty)$ and $\l\in
(-\infty,0]$. For simplicity one can think to \eqref{1} with $\l=0$.\\
It is well known that problem (\ref{1}) has a radial solution for any
$p\in (1,+\infty)$ (see \cite{KW}), and that this solution is unique
if
  $\l\in(-\infty,0]$ (see \cite{T} and \cite{F}). We will denote by $\up$ this
radial solution and  by $\cS$ the curve of radial
solutions of (\ref{1}) in the product space $ (1,+\infty)\times
C^{1,\a}_0(\overline A)$, where $C^{1,\a}_0(\overline A)$ is  the set of continuous differentiable  functions
 on $\overline A$ which vanish on  $\de A$ and whose first order
 derivatives are H\"older continuous with exponent $\a$. 
In other words:
\begin{equation}\label{1.2}
\cS:=\{(p,u_p)\in  (1,+\infty)\times C^{1,\a}_0(\overline A)\, \hbox{ such that }\up
    \hbox{ is the radial}
\hbox{ solution of (\ref{1})}\}.
\end{equation}
In this paper we study the nonradial solutions that bifurcate from the curve $\cS$ as the exponent $p$ varies.
Let us recall that 
a point $(p_k,u_{p_k})\in \cS$
is a {\em  nonradial bifurcation point} if in  every neighborhood of
$(p_k,u_{p_k})$  in $(1,+\infty)\times C^{1,\a}_0(\overline A)$
there exists a nonradial solution $(p,v_p)$ of (\ref{1}).\\
In the paper  \cite{GGPS} the
authors show that
there exists a
sequence of values of the exponent $p_k$ such that $p_k\to +\infty$
and $(p_k,u_{p_k})$ is a nonradial bifurcation point for $\cS$.\\
These values $p_k$ are found considering the linearized equation at the radial solution $u_p$. In  \cite{GGPS} it is shown that
$\up$ is degenerate if and only if, for some $k\geq 1$,
\begin{equation}\label{1.3a}
\a_1(p)+k(N-2+k)=0
\end{equation}
where $\a_1(p)$ is the first eigenvalue of the one-dimensional operator
\begin{equation}\label{1.3}
\widehat{L}_p(v)=r^2\left(-v'' -\frac{N-1}r v'-p\up
^{p-1}v-\l v\right)
\end{equation}
in the space of functions of $H^1_0(a,b)$. From the analyticity of $\a_1(p)$ with respect to $p$ and from the asymptotic behavior of $\a_1(p)$ as $p\to 1$ and as $p\to +\infty$ it is proved in \cite{GGPS} that for any $k\geq 1$ the quantity $\a_1(p)+k(N-2+k)$ changes sign as $p$ varies in $(1,+\infty)$. Each time $\a_1(p)+k(N-2+k)$ changes sign an eigenvalue of the linearized operator changes sign so that the Morse index of the radial solution changes. 
Let us call  {\em{Morse index changing
  points}} the pairs $(p_k,u_{p_k})\in \cS$ such that the Morse index
of the radial solution $u_p$ changes at $p_k$. These points are characterized
by
\begin{equation}\label{1.5}
\left(\a_1(p_k+\d)+ k(N-2+k)  \right)\left(\a_1(p_k-\d)+ k(N-2+k)\right)<0\quad
\hbox{ for }\d\in (0,\d_0)
\end{equation}
for some $\d_0>0$
and for some $k\geq 1$. 
In \cite{G} it is shown that if $(
p_k,u_{p_k})$ is a Morse index changing point there exists a
continuum, $\cC(p_k)\subset (1,+\infty)\times C^{1,\a}_0(\overline A)$, of nonradial solutions 
bifurcating from that point. 
This continuum $\cC(p_k)$ obeys
the so-called
Rabinowitz alternative,
(see Theorem 3.3 in \cite{G}), i.e. either $\cC(p_k)$
is unbounded in $(1,+\infty)\times C^{1,\a}_0(\overline A)$ or it must meet the curve of radial solutions $\cS$
in another Morse index changing point. Here we are able to prove that the second alternative is not possible when $k=1$ or $2$ and $N\geq 3$. Our main result is the following:
\begin{theorem}\label{t1}
If $N\geq 3$, there exist at least two exponents $p_1,p_2\in(1,+\infty)$ such that
$(p_k,u_{p_k})$ is a nonradial bifurcation point for the curve $\cS$,
related by (\ref{1.3}) to $k=1$ and $k=2$ at which the continuum of bifurcating solutions
$\cC(p_k)$ is unbounded in $(1,+\infty)\times C^{1,\a}_0(\overline A) $.
\end{theorem}
\noindent The unboundedness of the bifurcating branch in Theorem \ref{t1} implies or that the branch exists for every $p>p_k$ (giving a multiplicity result for problem \eqref{1}) or that the solutions along the branch make blow-up in the $C^{1,\a}$-norm at some exponent $\bar p\geq \frac{N+2}{N-2}$. Since problem \eqref{1} can be supercritical we cannot exclude that the branch exists only for a fixed value of the exponent $p$, even if we do not think this is the case. So we think that the behavior of the solutions along these unbounded branches deserves to be investigated further.\\ 
\noindent This result is a first attempt to separate all the branches generated by spherical harmonics ($O(N-1)$-invariant), related by \eqref{1.5} to a different eigenvalue $\mu_k=k(N-2+k)$ of the Laplace Beltrami operator on the $(N-1)$-dimensional sphere. Theorem \ref{t1} says that the branches generated by the spherical harmonics corresponding to $\mu_1=N-1$ and $\mu_2=2N$ ($k=1,2$) are separated from the others. The proof relies on the fact that functions which are $O(N-1)$-invariant, can be written as functions which depend only on $r$ and $\theta$ in radial coordinates, see the Appendix for details. Then to separate the continuum $\cC(p_1)$ and $\cC(p_2)$ from the others we introduce two different cones in $C^{1,\a}_0(\overline A)$ and we set our problem on these cones.\\
The cones $\mathcal{K}^1$ and $\mathcal{K}^2$ are defined in Section \ref{s3}, see \eqref{3.1} and \eqref{3.1-bis} and their definition is completely new. To define $K^1$  
we look for solutions which are non increasing with respect to the angle $\theta$ on the interval $[0,\pi]$. This property is preserved along the branch $\cC(p_1)$, as proved in Section \ref{s3}, and allows to distinguish the branch generated by $k=1$ from the others. Functions with this type of symmetry are said Foliated Schwarz symmetric and arise when looking for solutions with low Morse index (see \cite{GPW} as an example). Indeed in our case they arise when the Morse index of the radial solution $u_p$ goes from $1$ to $N+1$, see \eqref{morse-index}.\\ 
\noindent To define $\mathcal{K}^2$ we consider $O(N-1)$ invariant functions which are non increasing with respect to the angle $\theta$ on the interval $[0,\frac{\pi}2]$ and which are even in $z=\cos\theta$. Again this property is preserved along the continuum $\cC(p_2) $ and it is enough to exclude the branches bifurcating from exponents related by \eqref{1.3a} to $\mu_k$ with $k\neq 2$.\\
As said before this is a first attempt to separate branches of nonradial solutions when $N\geq 3$. The definition of the cones $\mathcal{K}^1$ and $\mathcal{K}^2$ is suggested by the shape of the $O(N-1)$-invariant spherical harmonics given explicitly in the Appendix.  We believe that it should be possible to separate all the branches generated by different spherical harmonics investigating in a deeper way their properties.\\
\noindent The cones $\mathcal{K}^1$ and $\mathcal{K}^2$, introduced here, separate the first spherical harmonics and can be used to distinguish solutions in a radially symmetric domain. The same method can be used to separate branches of nonradial, $O(N-1)$-invariant functions in other settings, for example in the case of the exterior of the ball, see \cite{GP} or in the case of the critical H\'enon problem in $\R^N$, see \cite{GGN}. We believe that also in these cases the cones $\mathcal{K}^1$ and $\mathcal{K}^2$ can give the unboundedness of the bifurcating branch of non radial solutions. Another application of these cones can be, for instance, to reduce the dimension of the kernel of the linearized equation to some problems, see \cite{GGT} as an example.\\[.5cm]
The problem to separate branches of solutions generated by different values of $k$  was solved in dimension $N=2$ by Dancer in the paper \cite{DA1}. Using the fact that the spherical harmonic functions associated to the eigenvalues $\mu_k$ are periodic with with period $\frac{2\pi}k$ when $N=2$ Dancer introduced some suitable cones $\mathcal{K}^k$ of periodic functions in which only the $k$-th spherical harmonic lies. This allows to separate branches of solutions related to different values of $\mu_k$ and can give also multiplicity results, see \cite{GGN2} as an example.  \\ 
Using exactly the same cones $\cK^n$ of Dancer we have the following result:
\begin{corollary}\label{c1}
If $N=2$, for any $n\in \N$ there exists an exponent $p_n\in(1,+\infty)$ such that the continuum $\cC(p_n)$ is unbounded in
$C^{1,\a}_0(\overline A) $. Moreover $\cC(p_n)\cap \cC(p_m)=\emptyset$ if $n\neq m$. Finally
for any $p>p_n$ there exist at least $n$ nonradial positive solutions
of (\ref{1}).
\end{corollary} 
The multiplicity result for $N=2$ follows since problem \eqref{1} is subcritical in $\R^2$ and cannot be obtained for $N\geq 3$ in an easy way.\\

The paper is organized as follows: in section \ref{s2} we
introduce all the notations. In section \ref{s3} we
prove Theorem \ref{t1} and Corollary \ref{c1}. Finally in the Appendix
we derive some properties of functions which are $O(N-1)$-invariant 
and we write explicitly the $O(N-1)$-invariant spherical harmonics.

\sezione{Notations and preliminary results}\label{s2}
The starting point in the study of bifurcation is the analysis  of the 
degeneracy points to (\ref{1}). To this end we consider the linearized equation at the radial solution $u_p$, i.e. 
\begin{equation}\label{l}
\left\{\begin{array}{ll}
-\Delta v-p\up^{p-1}v-\l v=0 & \hbox{ in }A\\
v=0 & \hbox{ on }\de A.
\end{array}\right.
\end{equation}
It is proved in \cite{GGPS} (see Lemma 2.3) that equation \eqref{l} admits a nontrivial solution if and only if
\begin{equation}\label{2.1}
\a_1(p)+\mu_k=0, \quad \hbox{ for some } k\geq 1,
\end{equation}
where $\a_1(p)$ is the first eigenvalue of the one-dimensional operator $\widehat L_{p}$ defined in \eqref{1.3}
and $\mu_k=k(N-2+k)$,
$k=0,1,\dots$ are the eigenvalues of the Laplace-Beltrami operator
$-\Delta_{S^{N-1}}$ on the sphere $S^{N-1}$.\\
Moreover  the solutions $v$ of the
linearized equation (\ref{l}) corresponding to a
degeneracy point $p_i$ can be written as
\begin{equation}\label{au}
v(x)=w_{1,p_i}(|x|)\phi_k\left(\frac x{|x|}\right)
\end{equation}
where $w_{1,p_i}(r)$ is the first positive eigenfunction of $\widehat L_{p_i}$
and $\phi_k$ is an eigenfunction of the Laplace-Beltrami
operator on  $S^{N-1}$ relative to the eigenvalue $\mu_k$. \\
As explained in
\cite{GGPS}, the Morse index of the radial solution $u_p$ that we denote by $m(p)$ depends only
on the sign of the sum $\a_1(p)+\mu_k$ for  $k\geq 1$ and precisely is given by (recall that $\a_1(p)<0$ for any $p$)
\begin{equation}\label{morse-index}
m(p) = \sum_{0 \leq  j<\frac{2-N}2+\frac1 2 \sqrt{(N-2)^2-4\a_1(p)}
\atop_{j\  integer}} \frac{(N+2j-2)(N+j-3)!}{(N-2)!\,j!}.
\end{equation}
\noindent In \cite{GGPS} we
proved the following result:
\begin{theorem}\label{t02}
The Morse index changing points are nonradial bifurcation points for (\ref{1}). Moreover the
exponents $p_i$ of these
points can be arranged in a
sequence that diverges to $+\infty$.
\end{theorem}
\noindent
To introduce the global bifurcation result obtained in \cite{G} we let $X$ 
be the subspace of $C^{1,\a}_0(\bar A)$
given by the functions which are $O(N-1)$-invariant, i.e. 
\begin{equation}\label{2.3}
X:=\{v\in C^{1,\a}_0(\overline A) \, \,
\hbox{s.t. }v(x_1,\dots,x_N)=v(g(x_1,\dots,x_{N-1}),x_N)\,\atop \hbox{
  for any } g\in O(N-1)\}
\end{equation}
where $O(N-1)$ is the orthogonal group in  $\R^{N-1}$, and
\begin{equation}\label{tpv}
\begin{array}{lrlc}T(p,v):&\hbox{ } (1,+\infty)\times X&\rightarrow&\hbox{  }X\\
&(p,v)\,\,\,\,\,&\mapsto & \left(-\D 
-\l\right)^{-1}\left( |v|^{p-1}v\right).
\end{array}
\end{equation}
$T$  is a compact operator
for fixed $p$ and is continuous with respect to $p$. We let
$S(p,v):=v-T(p,v)$. Then any solution of (\ref{1}) can be found as a
solution of $S(p,v)=0$ such that $v\geq 0$ in the annulus $A$.
Let us denote by $\Sigma$ the closure in $(1,+\infty)\times X$ of the set of
solutions of $S(p,v)=0$ different from $\up$,  i.e.
\begin{equation}\label{2.5}
\Sigma:=\overline{\{(p,v)\in (1,+\infty)\times X\,,\, S(p,v)=0\, ,\, v\neq \up\}}.
\end{equation}
If $(p_k,u_{p_k})\in \mathcal{S}$ is a nonradial bifurcation point, then
$(p_k,u_{p_k})\in \Sigma$. For  $(p_k,u_{p_k})\in \Sigma $
we will call $\cC(p_k)\subset \Sigma$
the closed  connected component of $\Sigma$  which contains
$(p_k,u_{p_k})$ and is maximal with respect to
the inclusion.\\
In \cite{G} we proved the following result:
\begin{theorem}\label{t12}
Let $(p_k,u_{p_k})$ be a Morse index changing point and let $\cC(p_k)$ as defined before. 
Then
either
\begin{itemize}
\item[a)] $\cC(p_k)$ is unbounded in $(1,+\infty)\times X $, or 
\item[b)] for some $h\neq k$, $(p_h,u_{p_h})$ is a
  Morse index changing point and $(p_h,u_{p_h})\in \cC(p_k)$.
\end{itemize}
\end{theorem}
\noindent Now we are in position to prove our first new result.
\begin{proposition}\label{p1}
Let $(p_k,u_{p_k})$ be a Morse index changing point and let $\cC(p_k)$
be as defined before. If $\cC(p_k)$ is bounded,
the number of the Morse index changing points in $\cC(p_k)$ including $
(p_k,u_{p_k})$ is even.
\end{proposition}
\noindent This result is based on an improved version of the Rabinowitz alternative due to Ize (see \cite{N}). We report the proof for completeness, see also \cite{AG}.\\
\begin{proof} If
$\cC(p_k)$ is bounded then $b)$ of Theorem \ref{t12} holds and $\cC(p_k)$ must meet the curve $\cS$ at least in one point  $ (p_h, u_{{p_h}})$  such  that $p_h$ is a degeneracy point, i.e. satisfies \eqref{2.1}. But it can meet the curve $\cS$ also in other bifurcation points. Since $\cC(p_k)$ is bounded and the exponents $p_i$ of the bifurcation points must satisfy \eqref{2.1} then $\cC(p_k)$ can meet $\cS$ at most in finitely many 
bifurcation
points $(p_i,u_i)$, $i=1,\dots,n$ with  $p_1<p_2<\dots<p_n$.
Arguing as in the proof of Theorem 3.3 in \cite{G},
we can find  a bounded open set
$\cO\subset (1,+\infty)\times X$ such that $\cC(p_k)\subset \cO$, $\de
O\cap \Sigma=\emptyset$ where $\Sigma$ is as defined in \eqref{2.5}. Moreover we can  assume that $\cO$ does not contain points $(p,\up)$ if
$|p-p_i|\geq \e_0$ for $i=1,\dots, n$ and $\e_0>0$ such that there are not degeneracy points in
$\cup_{i=1}^n (p_i-2\e_0,p_i+2\e_0)$. 
For $\cO$ as above and $r>0$, consider the map
$$\begin{array}{llll}S_r(p,v):&\hbox{ }\bar \cO&\rightarrow&\hbox{  } X\times \R\\
&(p,v)&\mapsto &\left(S(p,v),\nor v-\up\nor_X^2 -r^2\right)
\end{array}$$
where $\nor \cdot\nor_X$ stands for the usual norm in the space $C^{1,\a}_0(A)$.
Now, $\mathit{deg}\left( S_r(p,v),\cO,(0,0)\right)$ is defined since
on $\de \cO$ there are no  solutions of $S(p,v)=0$ different from the
radial solution $u_p$, and hence
$0=\nor v-\up \nor_X<r$ for such any solution. Furthermore the degree is
independent of $r>0$. For large $r$, $S_r(p,v)=(0,0)$ has no solutions in
$\cO$, and hence has degree zero. On the other hand, for small  $r$, if
$(p,v)$ is a solution of $S_r(p,v)=(0,0)$, then $\nor v-\up\nor_X =r$, and hence
$p$ is close to one of the $p_i$, $i=1,\dots,n$. But then the sum of
local degrees of $S_r$ in the neighborhoods of each of the $p_i$ is
equal to zero, so that
\begin{equation}\label{2.6}
0=\sum_{i=1}^n \mathit{deg}\left( S_r(p,v),\cO\cap
B_{r}(p_i,u_{p_i}),(0,0)\right).
\end{equation}
In particular we choose $r<\e_0$ for $\e_0$ defined as before.
In order to compute the degree of $S_r(p,v)$ in $\cO\cap B_{r}(p_i,u_{p_i})$ we
use again the homotopy invariance of the degree. Let us define
$$S_r^t(p,v)=\left(S(p,v), t(\nor v-\up\nor_X^2
-r^2)+(1-t)(2p_ip-p^2-p_i^2+r^2)\right)$$
for $t\in [0,1]$.
As before $\mathit{deg}\left( S_r^t(p,v), \cO\cap B_{r}(p_i,u_{p_i}),
  (0,0)\right) $ is well defined  since there are no solutions on the
  boundary if $r$ is small (recall that $u_{p_i\pm r}$ are isolated if
  $r<\e_0$). Moreover the degree is independent of
  $t$. For $t=1$ we have $S_r^1(v,p)=S_r(p,v)$, while for $t=0$,
$S_r^0(p,v)=\left(S(p,v),2p_ip-p^2-p_i^2+r^2\right)$ and
\begin{eqnarray}
&& \mathit{deg}\left( S_r^0(p,v),\cO\cap B_{r}(p_i,u_{p_i}),
  (0,0)\right)\nonumber\\
&&= \mathit{deg} \left(
  S(p,v),\cO\cap B_{r}(p_i,u_{p_i}) ,0\right)\cdot \mathit{deg}\left( 2p_ip-p^2-p_i^2+r^2,
  \{|p-p_i|<r\},0\right).\nonumber
\end{eqnarray}
Now
$$\mathit{deg}\left( 2p_ip-p^2-p_i^2+r^2,
  \{|p-p_i|<r\},0\right)=1$$
for $p=p_i-r$ while
$$\mathit{deg}\left( 2p_ip-p^2-p_i^2+r^2,
  \{|p-p_i|<r\},0\right)=-1$$
for $p=p_i+r$. This implies that
\begin{eqnarray}
&&\mathit{deg}\left( S_r(p,v),\cO\cap
B_{r}(p_i,u_{p_i}),(0,0)\right)=\nonumber\\
&&\mathit{deg}\left(
S(p_i-r,\cdot),\cO_{p_i-r},0\right) - \mathit{deg}\left(
S(p_i+r,\cdot),\cO_{p_i+r},0\right)\nonumber\\
&&=(-1)^{m(p_i-r)}-(-1)^{m(p_i+r)}\nonumber
\end{eqnarray}
where as in \cite{G} we denote by $\cO_p$ the set $\{v\in \cO\, :\,
(p,v)\in \cO\}$.\\
We conclude that if $(p_i, u_{p_i})$ is a Morse index changing
point then
$$\mathit{deg}\left( S_r(p,v),\cO\cap
B_{r}(p_i,u_{p_i}),(0,0)\right)=\pm 2$$ while if $(p_i, u_{p_i})$ is
not a Morse index changing point then
$$\mathit{deg}\left( S_r(p,v),\cO\cap
B_{r}(p_i,u_{p_i}),(0,0)\right)=0.$$ Since the nonzero terms in
(\ref{2.6}) correspond  only to the Morse index changing points, and since these
terms add up to zero, there must be an even number of Morse index
changing points.
\end{proof}
We let $(p_k,u_{p_k})$ be a bifurcation point, corresponding, via  (\ref{1.3a})
to the eigenvalue $\mu_k$. Then the
linearized operator $L_{p_k}$ at $u_{p_k}$ has, up to a constant
multiple, a unique solution in $X$, see \cite{SW1}, which has the form given by
(\ref{au}). We let $w_k$  be this unique normalized (in the
$L^{\infty}$-norm) eigenfunction of $L_{p_k}(v)=0$ in $X$. Now we want to prove the following result:
\begin{proposition}\label{p2}
There exists $\rho_0>0$ such that if 
$(p,v)\in \left(\cC(p_k)\setminus\{(p_k,u_{p_k})\}\right)\cap
B_{\rho}(p_k,u_{p_k})$, then $v-\up=\a_p w_k+r_p$, where $w_k$ is as
before and $\a_p\to 0$ as $p\to p_k$
and $r_p=o(\a_p)$ as $p\to p_k$.
\end{proposition}
\begin{proof}
With the previous notations we let $l_k$ be an
element of the dual space $X'$ of $X$,  such that $<l_k,w_k>=1$, where
$<\cdot,\cdot>$ denotes the duality between $X$ and $X'$. This element
$l_k$ exists thanks to the Hahn-Banach Theorem. Let $X_k=\{z\in
X\,\hbox{ such that } <l_k,z>=0\}$. Since $\mathit{dim}\left(\mathit{Ker}\,
L_{p_k}\right)$ is finite, then $\mathit{Ker}\, 
L_{p_k}$ is complemented in $X$, see for example \cite{M} pag
300. Hence we can  decompose  $X=\R\oplus
X_k$, and every $z\in X$ can be written as  $z=\a w_k+r$, with
$\a=<l_k,z>$ and $r\in X_k$. \\ 
For $\eta\in(0,1)$ and $\gamma>0$, we let 
$$K_{\gamma,\eta}=\{(p,v)\in X\hbox{ such that }|p-p_k|<\gamma \hbox{
  and }|<l_k,v-\up>|>\eta \nor v-\up\nor_{\infty}\}$$
where $\nor\cdot \nor_{\infty}$ denotes the usual $L^{\infty}$-norm. 
We want to prove, first, that, for any $\eta$ and any $\gamma$, there exists $\rho_0>0$ such that for all $\rho<\rho_0$ we have
$\left(\cC(p_k)\setminus\{(p_k,u_{p_k})\}\right)\cap
B_{\rho}(p_k,u_{p_k})\subset K_{\gamma,\eta}$.  Here
$B_{\rho}(p_k,u_{p_k})$ denotes the ball of radius $\rho$ in the product space
$(1,+\infty)\times X$. If there is not such a $\rho_0$, there exist
  sequences $\rho_n\to 0$ and $(p_n,v_n)\in \left(\cC(p_k)
  \setminus\{(p_k,u_{p_k})\}\right)\cap B_{\rho_n}(p_k,u_{p_k})$ such that
  $|p_n-p_k|\leq \rho_n<\gamma$, $v_n-u_{p_n}\to 0$ in $X$ and $|\a_n|:=|<l_k,
  v_n-u_{p_n}>|\leq \eta \nor v_n-u_{p_n}\nor_{\infty}$.  Letting
  $z_n=\frac{v_n-u_{p_n}}{\nor v_n-u_{p_n}\nor_{\infty}}$ we have that $z_n$
  satisfies
\begin{equation}\label{**}
\left\{ \begin{array}{ll}
-\Delta z_n=h_n(x)z_n +\l z_n& \hbox{ in }A\\
z_n=0 & \hbox{ on }\de A
\end{array}\right.
\end{equation}
where
\begin{equation}\label{2.10}
h_n(x)=p_n\int_0^1 \left( u_{p_n}+t(v_n-u_{p_n})\right)^{p_n-1}dt
\end{equation}
and $h_n(x)\to p_k u_{p_k}^{p_k-1}$ in $X$ as $n\to +\infty$. Then $\nor
z_n\nor_{\infty}=1$ and $z_n\to z$ uniformly in $\bar A$ where $z$ is a
solution of
\begin{equation}\label{2.11}
  \left\{\begin{array}{ll}
-\Delta z=p_k u_{p_k}^{p_k-1}z+\l z &\hbox{ in }A\\
z=0 & \hbox{ on }\de A
\end{array}\right.
\end{equation}
such that $\nor z\nor_{\infty} =1$. This implies that either $z=w_k$ or $z=-w_k$. Moreover
$\frac{\a_n}{\nor v_n-u_{p_n}\nor_{\infty}}:=\frac {|<l_k,v_n-u_{p_n}>|}{\nor v_n-u_{p_n}\nor_{\infty}}\to |<l_k,\pm
w_k>|=1>\eta$ and we get a contradiction. \\
Thus there exists a
$\rho_0>0$ as above. Using the previous decomposition $X=\R\oplus
X_k$, we have that $v-\up=\a_p w_k+r_p$ where $\a_p:=<l_k,v-\up>$ and
$r_p:=v-\up-\a_pw_k\in X_k$. \\
Now, if $p_n\to p_k$ then $|\a_n|:=|\a_{p_n}|=|<l_k,v_n-u_{p_n}>|=\nor
v_n-u_{p_n}\nor_{\infty}|<l_k,z_n>|=\nor
v_n-u_{p_n}\nor_{\infty}(1+o(1))$. This implies that $\a_n\to 0$ as
$p_n\to p_k$. Finally 
$r_n:=r_{p_n}=v_n-u_{p_n}-\a_n w_k$ so that 
\begin{eqnarray}
&&\nor r_n\nor_{\infty} \leq \nor v_n-u_{p_n}\nor_{\infty}+|\a_n|\,\nor w_k\nor_{\infty}<\frac 1{\eta}\,\,
  |<l_k,v_n-u_{p_n}>|+|\a_n|\nonumber\\
&& =|\a_n|+\frac 1{\eta}\,|\a_n|\,\,
  |<l_k,w_k>|+\frac 1{\eta}\,\,
  |<l_k,r_n>|=\frac 1{\eta}\,
  |\a_n|+|\a_n|.\nonumber
\end{eqnarray}
This shows that $r_n\to 0$ as $p_n\to p_k$. Finally 
$$\Big|<l_k,\frac {u_n-u_{p_n}}{\nor v_n-u_{p_n}\nor_{\infty}}>\Big|=\frac {\a_n}{\nor
  v_n-u_{p_n}\nor_{\infty}}\to 1$$
so that 
$$\frac {r_n}{\nor v_n-u_{p_n}\nor_{\infty}}=\frac{v_n-u_{p_n}-\a_n w_k}{\nor
  v_n-u_{p_n}\nor_{\infty}}\to w_k-w_k=0.$$
This shows that $r_n=o(|\a_n|)$ as $p_n\to p_k$ and finishes the proof.
\end{proof}
\noindent This proposition gives us the behavior of the branch of
  solutions $\cC(p_k)$ near a bifurcation point $(p_k,u_{p_k})$.
\sezione{Proof of the main result.}\label{s3}
In this section we want to prove Theorem \ref{t1}. As before we
consider functions which are $O(N-1)$-invariant, i.e. the space $X$
defined in (\ref{2.3}). It is easy to see (see the Appendix for the
details) that if $(\rho,\phi_1,\dots,\phi_{N-2},\theta)$ with
$a\leq \rho\leq b$ and $\phi_i\in[0,2\pi]$ for $i=1,\dots,N-2$ and
$\theta\in[0,\pi]$, are the radial coordinates in $\R^N$, then a
$O(N-1)$-invariant function in $\R^N$ can be written as a function
which depends only on $\rho$ and
$\theta$. \\
We introduce the following cones:
\begin{equation}\label{3.1}
\mathcal{K}^1=\left\{\begin{array}{l}v\in C^{1,\a}(\bar A)\,\hbox{ s.t. }v \hbox{ is }
O(N-1)-\hbox{invariant}, \, v\geq 0 \hbox{ in }A,\\ 
v=v(\rho, \theta)  \hbox{ in radial coordinates for } (\rho, \theta)\in A,\\
 v(\rho,\theta) \hbox{ is even in } \theta \hbox{ and }v(\rho,\pi+\theta)=v(\rho,\pi-\theta)\\ 
v(\rho,\theta) \hbox{ is non increasing in }\theta \hbox{ for } \theta\in [0,\pi], \rho\in[a,b] 
\end{array}\right\}
\end{equation}
and
\begin{equation}\label{3.1-bis}
\mathcal{K}^2=\left\{\begin{array}{l}v\in C^{1,\a}(\bar A)\,\hbox{ s.t. }v \hbox{ is }
O(N-1)-\hbox{invariant}, \, v\geq 0 \hbox{ in }A,\\ 
v=v(\rho, \theta)  \hbox{ in radial coordinates for } (\rho, \theta)\in A,\\
v(\rho,z) \hbox{ is even in } z, \hbox{ where  } z=\cos\theta, \hbox{ for } z\in[-1,1]\\
v(\rho,\theta) \hbox{ is non increasing in }\theta \hbox{ for } \theta\in [0,\frac \pi 2], \rho\in[a,b]\ 
\end{array}\right\}.
\end{equation}
As said in the Introduction functions that belong to $\mathcal{K}^1$ are Foliated Schwarz symmetric, see \cite{GPW} for some comments on this type of symmetry.\\
Since the terminology is not uniform in the literature we recall that
$W$ is a cone in $X$, if $W$ is a closed convex set in $X$ such that
$\g W\subseteq W$ for any $\g\geq 0$ and $W\cap -W=\{\emptyset\}$.\\
First we can prove the following result:

\begin{lemma}\label{l31} For any $p\in(1,+\infty)$ the map
  $T(p,-):X\rightarrow X$, defined in (\ref{tpv}), maps the cone $\cK^i$ into
  itself.
\end{lemma}
\begin{proof}
Suppose $g\in \cK^i$, then the function $g^p\in \cK^i$ for $i=1,2$. We have that  $T(p,g)=w$ if $w$ is a solution to
\begin{equation}\label{3.2}
\left\{\begin{array}{ll}
-\Delta w-\l w =g^p & \hbox{ in }A\\
w=0&\hbox{ on }\de A.
\end{array}\right.
\end{equation}
First, since $g\geq 0$ in $A$ and $\l\le 0$, the Maximum principle implies  $w\geq 0$ in $A$. 
We already know that the
map $T(p,-)$ is invariant with respect the action of the group
$O(N-1)$, so that  $T(p,-)$ maps the space $X$ into itself. Thus $w$ is $O(N-1)$-invariant and, as
previously said, we can write $w=w(\rho, \theta)$ with the properties $w(\rho,\theta)=w(\rho,-\theta)$ and $w(\rho,\pi+\theta)=w(\rho,\pi-\theta)$ for every $(\rho,\theta)\in A$, see the Appendix for these details. 
Moreover, if $g$ is even in $z$ then $g^p$ is even in $z$ and so also $w=T(p,g)$ is even in $z$. 
Exploiting
this fact we can rewrite (\ref{3.2}) in radial coordinates, getting
that $w$ satisfies
\begin{equation}\label{3.3}
\left\{\begin{array}{ll}
-\frac{\de ^2 w}{\de \rho^2}-\frac{N-1}{\rho} \frac{\de w}{\de
  \rho}-\frac 1{\rho^2} \frac{\de ^2 w}{\de \theta^2}-\frac{N-2}{\rho^2}\cot \theta\frac{\de w}{\de \theta}-\l w=g^p & \hbox{ in }A\\
w=0&\hbox{ on }\de A
\end{array}\right.
\end{equation}
Differentiating with respect to $\theta$ we get that
$w_{\theta}:=\frac{\de w}{\de \theta}$ satisfies
$$-\frac{\de ^2 \wth}{\de \rho^2}-\frac 1{\rho^2}\frac{\de ^2 \wth}{\de \theta^2}-\frac{N-1}{\rho} \frac{\de \wth}{\de
  \rho}-\frac{N-2}{\rho^2}\cot \theta  \frac{\de \wth}{\de \theta}
+\frac{N-2}{\rho^2\sin^2\theta}\wth-\l \wth=pg^{p-1}\frac{\de
  g}{\de \theta}$$
for $\rho\in [a,b]$ and $\theta\in [0,\pi]$.  This is a second order operator uniformly elliptic in $[a,b]\times [0,\pi]$. The coefficient of the linear term is $c(\rho,\theta)=\frac{N-2}{\rho^2\sin^2\theta}-\l> 0$ in $(a,b)\times (0,\pi)$ and it is bounded in every closed ball in $(a,b)\times (0,\pi)$. Also the coefficients of the first order terms, i.e. $\frac{N-1}{\rho}$ and $\frac{N-2}{\rho^2}\cot \theta$ are bounded on every closed ball in $(a,b)\times (0,\pi)$. Then, we consider first the case of $g\in\cK^1$, the maximum principle applies since $\frac{\de g}{\de \theta}\leq 0$ for $(\rho,\theta)\in (a,b)\times (0,\pi)$ and implies that $\wth$ reaches its maximum on the boundary of $(a,b)\times (0,\pi)$, see \cite{PW} pag 64. Then, the boundary conditions $v(a,\theta)=v(b,\theta)=0$ for every $\theta$ imply that $\wth(a,\theta)=\wth(b,\theta)=0$ for every $\theta\in [0,\pi]$. Finally the symmetry assumptions on $v$ imply, in turn, that $\wth(\rho,0)=\wth(\rho,\pi)=0$, see the Appendix for details, so that $\wth\leq 0$ in $[a,b]\times [0,\pi]$. This implies that $v\in \cK^1$ and concludes the proof in the case of $g\in \cK^1$.\\
Now assume $g\in \cK^2$. As said before $w=T(p,g)$ is even in $z$. By assumptions we have that $\frac{\de g}{\de \theta}\leq 0$ for $(\rho,\theta)\in (a,b)\times (0,\frac \pi 2)$. Again we can apply the maximum principle getting that $w_{\theta}$ reaches its maximum on the boundary of $(a,b)\times (0,\frac \pi 2)$. As before $\wth(a,\theta)=\wth(b,\theta)=0$ for every $\theta\in [0,\frac \pi 2]$ and  $\wth(\rho,0)=0$ for every $\rho\in(a,b)$. Finally since $w$ is even in $z$ we get that $w(\rho,\cos\theta)=w(\rho,-\cos\theta)$ and this implies $\wth(\rho,\frac \pi 2)=0$ for any $\rho\in(a,b)$. Then we have $\wth\leq 0$ in  $[a,b]\times [0,\frac \pi 2]$ showing that $w\in \cK^2$. This concludes the proof of the Lemma.
\end{proof}
Before proving the main result we need some notations, following
\cite{DA0}. Given a cone $\cK$ and a point $u\in\cK$ we let $\cK_u:=\{
v\in X\,:\, u+tv\in\cK\hbox{ for some }t>0\}$ and $S_u:=\{ v\in
\overline{\cK}_u\,:\, -v\in \overline{\cK}_u\}$. Then we have that, if
$\up$ is a radial solution of (\ref{1}) then $\frac {\de \up}{\de \theta}\equiv 0$ in $A$  and hence $\overline{\cK}^1_{\up}=\{v\in X\,:\,v=v(\rho, \theta)  \hbox{ in radial coordinates for } (\rho, \theta)\in A,\
 v(\rho,\theta) \hbox{ is even in } \theta \ , \ \frac {\de v}{\de
  \theta}\geq 0 \hbox{ in }A\}$ while $\overline{\cK}^2_{\up}=\{v\in X\,:\,v=v(\rho, \theta)  \hbox{ in radial }\linebreak[2] \hbox{coordinates }\hbox{for } (\rho, \theta)\in A,
v(\rho,z) \hbox{ is even in } z=\cos\theta \ , \ \frac {\de v}{\de
  \theta}\geq 0 \hbox{ for any } \theta\in [0,\frac \pi 2], \rho\in[a,b]  \}$. This implies that, using the fact that $v\in \cK^2$ is even in $z$, $S^1_{\up}=S^2_{\up} =\{ v\in X\,:\hbox{ such that }v \hbox{ is radially
  symmetric}\}$.   See \cite{DA1} for details.\\ 

\begin{proposition}\label{p31}
Let $\up$ be a radial solution of (\ref{1}) which is
nondegenerate. Then for $i=1,2$
\begin{equation}\label{3.5}
\mathit{index}_{\cK^i}\left( I-T(p,-),\up\right)=\left\{\begin{array}{ll}
\pm 1 & \hbox{ if }\a_1(p)+\mu_i>0\\
\\
0& \hbox{ if }\a_1(p)+\mu_i<0
\end{array}\right.
\end{equation}
where $\a_1(p)$ is the first eigenvalue of the {\em radial} operator
defined in (\ref{1.3}) and $\mu_1=N-1$ and $\mu_2=2N$.
\end{proposition}  
\begin{proof}
To calculate the index of $I-T(p,-)$ in the cone $\cK^i$ at the radial solution $\up$ we
use Theorem 1 in \cite{DA0}. First we  observe that, with the previous
notations, we have that $\cK^i-\cK^i$ is dense in $X$. 
Moreover by assumptions we have that $\up$ is a fixed point
of $T(p,-)$ in $\cK^i$ and $T(p,\up)$ is differentiable at $\up$ with
$T'(p,\up)$ invertible, since $u_p$ is nondegenerate. We are in position to apply
Theorem 1 in \cite{DA0} getting that
$$\mathit{index}_{\cK^i}\left(
I-T(p,-),\up\right)=\left\{\begin{array}{l}
0 \ \hbox{ if }\a_1(p)+\mu_i<0\\
\\
 \mathit{index}_X\left(
I-T(p,-),\up\right)=\pm 1 \,\,\hbox{ otherwise. }
\end{array}\right.$$ 
This claim follows from \cite[Theorem 1]{DA0} if we check that 
$T'(p,\up)$ has an eigenvalue in $(1,+\infty)$
  with corresponding eigenvector in $\overline{\cK}^i_{\up}\setminus
  S_{\up}$ if and only if $\a_1(p)+\mu_i<0$, see also Lemma 2 in \cite{DA0} and the Remark after it.\\
This is equivalent to show that the linearized operator has a negative eigenvalue with eigenfunction in $\overline{\cK}^i_{\up}\setminus S_{\up}$.  In \cite{GGPS} it is shown that the solutions of the linearized equation have the form given in \eqref{au}. This result holds also for the eigenfunctions of the eigenvalue problem with weight associated to the linearized equation, see \cite{GGN} for a proof of this assertion. Then, if we restrict to the space $X$ we have that an eigenvalue of the linearized problem becomes negative each time $\a_1(p)+\mu_k$ becomes negative and the corresponding eigenfunctions have the form
given in \eqref{au}, i.e. is the product of a positive radial function for a $O(N-1)$-invariant spherical harmonic function. Using the characterization of $O(N-1)$-invariant spherical harmonics we have that the linearized operator has a negative eigenvalue in $\cK^1$ if and only if $\a_1(p)+\mu_1<0$, while the linearized operator has a negative eigenvalue in $\cK^2$ if and only if $\a_1(p)+\mu_2<0$, since the eigenfunction corresponding to the negative eigenvalue $\a_1(p)+\mu_1$ does not belong to $\cK^2$ (it is not even in $z$). This finishes the proof.
\end{proof}

\begin{proof} [Proof of Theorem \ref{t1}]
We prove the result in the case of the exponent $p_1$ related by \eqref{1.3a} to $\mu_1$. The case of the exponent $p_2$ related to $\mu_2$ follows in the same way substituting the cone $\cK^1$ with $\cK^2$.\\ 
Let $p_1,\dots,p_M$ be the Morse index changing points related by
(\ref{1.3a}) to the first eigenvalue $\mu_1$. We can repeat the proof
of Theorem 3.3 in \cite{G} using the cone $\cK^1$ instead of the space $X$. Hence we get, for any
$p_j$, $j=1,\dots,M$, the
existence of a continuum $\cC(p_j)$ of solutions of (\ref{1}) which lies in
the cone $\cK^1$. Further  this continuum either is unbounded in $\cK^1$
or it must
intersect the curve of radial solutions $\cS$ in another Morse index
changing point. Moreover the points at which $\cC(p_j)$ can intersect
the curve of radial solutions $\cS$ are related to the first
eigenvalue $\mu_1$. This follows since otherwise the continuum $
\cC(p_j)$ is not contained in $\cK^1$, see also Proposition
\ref{p2}. Repeating the proof of Proposition \ref{p1}, in the cone
$\cK^1$ we have that the number of  Morse index changing points which
belong to a bounded continuum $\cC(p_j)$ has to be even.
On the other hand the number of Morse index changing points
corresponding to the eigenvalue $\mu_1$, is odd, since
$\a_1(p)+\mu_1>0$ if $p$ is near $1$ while $\a_1(p)+\mu_1<0$ if $p$ is
large enough.\\
This implies the existence of a value $p_1$ such that
$\a_1(p_1)+\mu_1=0$ and
$\cC(p_1)$ is unbounded in X.
\end{proof}
\begin{remark}
We suspect that the equation $\a_1(p)+\l_k=0$ has only one solution,
but we are not able to prove it. In that case any degeneracy point
would be
a Morse index changing point and the branch of bifurcating solutions
would  be unbounded.
\end{remark}
Now we sketch the proof of Corollary \ref{c1}. Let $(\rho,\theta)$ be
the radial coordinates in $\R^2$. As said before the
proof of this result follows using the cones
\begin{eqnarray}
\cK^n&:=&\left\{ v\in C(A)\,:\, v\geq 0\hbox{ in }A,
v(\rho ,\theta)=v(\rho,\theta+\frac{2\pi}{n}) \hbox{ for }(r,\theta)\in A\,,
v\hbox{ is even in }\theta\,,\right.\nonumber\\
&& \left. v(\rho,\theta)\hbox{ is decreasing in
}\theta \hbox{ for }0<\theta<\frac{\pi}{n}\, ,\, a\leq \rho\leq b\hbox{
  and } v=0\hbox{ on }\de A\right\}\nonumber
\end{eqnarray}
introduced by Dancer in \cite{DA1}. 
For these cones $\cK^n$ the analogous of Lemma 1 and Theorem 1 in
\cite{DA1} holds. These cones allow to separate continua of solutions
of (\ref{1}) in $\R^2$ related by (\ref{1.3a}) to a different eigenvalue
$\mu_k$. Using, as in the proof of Theorem \ref{t1}, the Proposition
\ref{p1} then we have that, corresponding to any $k$, there exists at
least a continuum $\cC(p_k)$ which is unbounded in $X$. Finally since
we are in dimension 2 the solutions of (\ref{1}) cannot blow up at a
finite value $p^*$, so that the unbounded continuum $ \cC(p_k)$ has to
be defined for every $p>p_k$.

\sezione{Appendix}
\noindent{\bf The $O(N-1)$-invariant functions in $\R^N$.}\\[.1cm]
Let us consider the spherical coordinates in $\R^N$, $(\rho,
\phi_1,\dots,\phi_{N-2},\theta)$ where $\phi_i\in[0,2\pi]$,
$i=1,\dots,N-2$ and $\theta\in [0,\pi]$. As usual
$$\left\{\begin{array}{ll}
x_i=\rho \sin \theta H_i( \phi_1,\dots,\phi_{N-2}) & i=1,\dots,N-2\\
x_N=\rho \cos \theta
\end{array}
\right.$$
where $H_i$ are suitable functions.
We are interested in the $O(N-1)$-invariant functions, i.e. functions
$v$ such that 
$$v(x_1,\dots,x_N)=v(g(x_1,\dots,x_{N-1}),x_N)$$
for any $g\in O(N-1)$. By definition, a function which is $O(N-1)$-invariant depends only on $\rho'=\sqrt{x_1^2+\dots+x_{N-1}^2}$ and $x_N$. Then, in radial coordinates, since $ \rho'=\sqrt{\rho^2-x_N^2}=\sqrt{\rho^2(1-\cos^2\theta)}=\rho|\sin \theta|$ and $x_N=\rho\cos\theta$, $v$ can be written as a function which depends only on $\rho$ and $\theta$.\\
Moreover, $v$ must satisfy $v(\rho, \theta)=v(\rho,-\theta)$ and $v(\rho, \pi+\theta)= v(\rho, \pi-\theta)$. This assertion follows since $v$ must depend only on $\rho'$ and $x_N$ and as functions of $\rho, \theta$ they satisfy  $\rho'(\rho,-\theta)=\rho'(\rho,\theta)$, $x_N(\rho,-\theta)= x_N(\rho,\theta)$ and $\rho'(\rho,\pi+\theta)=\rho|\sin(\pi+\theta)|=\rho|\sin\theta|=\rho|\sin(\pi-\theta)|=\rho'(\rho,\pi-\theta)$, $x_N(\rho,\pi+\theta)=\rho\cos(\pi+\theta)=-\rho\cos\theta=\rho \cos(\pi-\theta)= x_N(\rho,\pi-\theta)$.\\
Then an $O(N-1)$-invariant function $v$ satisfies $v(\rho, \theta)=v(\rho,-\theta)$ and $v(\rho, \pi+\theta)= v(\rho, \pi-\theta)$ and, if $v\in C^1(A)$ then it is $C^1((a,b)\times [0,\pi])$ and it verifies
$\frac{\de v}{\de \theta}(\rho,0)=\frac{\de v}{\de \theta}(\rho,\pi)=0$ for any $\rho>0$. \\[1cm]
\noindent{\bf Some remarks on the $O(N-1)$-invariant spherical
  harmonic functions.}\\[.1cm]
From what we said before the $O(N-1)$-invariant spherical harmonics
can be written as functions which depend only on the variable $\theta$. Then, the $k$-th $O(N-1)$-invariant spherical harmonic satisfies
\begin{equation}\label{A1}
-\sin^2 \theta \frac {\de^2 \Phi_k}{\de \theta^2}
-(N-2)\sin\theta\cos\theta \,\,\frac {\de \Phi_k}{\de \theta} =\l_k 
\sin ^2\theta \,\, \Phi_k
\end{equation}
for $\theta\in (0,\pi)$. 
Letting $z=\cos\theta$ we get that $\Phi_k(z)$ satisfies
\begin{equation}\label{A2}
(1-z^2)\frac {\de ^2 \Phi_k}{\de z^2}-(N-1)z\frac {\de \Phi_k}{\de
    z}+\l_k \Phi_k=0
\end{equation}
for $z\in(-1,1)$. 
This is a {\em Sturm-Lioville problem}. Then we can say that
the $k$-th eigenfunction has k different zeros in $[-1,1]$. From the
Sturm Theorem between two consecutive zeros of $\Phi_k$ there is a
zero of $\Phi_{k+1}$.\\ 
Equation \eqref{A2} is the {\em{Jacobi equation}} with, using the usual notations for Jacobi, $\a=\b=\frac{N-3}2$ and $n=k$. Then the bounded solutions of \eqref{A2} are given, up to a constant multiple, by the Jacobi polynomials, that can be written, using the Rodrigues' formula 
\begin{equation}\label{A3}
P_k^{(\frac{N-3}2,\frac{N-3}2)}(z)=\frac{(-1)^k}{2^k k!}(1-z^2)^{-\frac{N-3}2}\frac{\de ^k}{\de z^k}\left( (1-z^2)^{k+\frac{N-3}2}\right)
\end{equation}
for $z\in (-1,1)$ and any $k\geq 0$.\\
If $\a=\b=0$, i.e. for $N=3$, the Jacobi polynomials reduce to the Legendre polynomials
\begin{equation}\nonumber
P_k(z)=\frac 1{2^k k!}\frac{\de ^k}{\de z ^k}\left( z^2-1\right)^k
\end{equation}
and, indeed for   $N=3$, (\ref{A2}) is the classical Legendre equation.\\
Then the $O(N-1)$-invariant spherical harmonics are, up to a constant
multiple, the functions:
$$\Phi_k(\theta)=P_k^{(\frac{N-3}2,\frac{N-3}2)}(\cos \theta)$$
for $\theta\in (0,\pi)$, where $P_k^{(\frac{N-3}2,\frac{N-3}2)}$ are the Jacobi Polynomials.\\
To give some examples we have 
$$\begin{array}{l}
\Phi_1(\theta)=\frac {N-1}2 \cos\theta,\\
\Phi_2(\theta)= \frac{N-1}8\left(N\cos^2\theta -1\right),\\ 
\Phi_3(\theta)=\frac 1{48}(N+3)(N+1)\left((N+2)\cos^3\theta -3\cos\theta\right)
\end{array}$$
This implies, in turn, that the unique $O(N-1)$-invariant spherical
harmonic $\Phi_1$ related to the first eigenvalue $\l_1$ is, up to a
constant multiple, 
$$\Phi_1(\theta)=\cos\theta.$$
It then follows that $\frac{\de \Phi_1}{\de \theta}=-\sin \theta\leq0$ in $[0,\pi]$ while for all the other spherical harmonics the derivative $\frac{\de \Phi_i}{\de \theta}$ must change sign in $[0,\pi]$. This can be seen since 
from the formulation (\ref{A2})
$\Phi_1(z)=z$ changes sign once in $(-1,1)$ so that any other solution 
$\Phi_i(z)$  changes sign at least twice in $(-1,1)$ and this implies
that the derivative $\frac{\de \Phi_i}{\de z}$ has to change sign in
$(-1,1)$ so that also $ \frac{\de \Phi_i}{\de \theta}$ has to change sign
in $(0,\pi)$.

\end{document}